\documentclass[a4paper,11pt]{article}
\usepackage{amsmath,amssymb,amsthm,amscd}
\usepackage{graphicx,mathrsfs}
\usepackage [latin1]{inputenc}

\newtheorem{Thm}{Theorem}[section]
\newtheorem{Lem}[Thm]{Lemma}

\newtheorem{Prop}[Thm]{Proposition}

\newtheorem{Def}[Thm]{Definition}
\newtheorem{Rem}[Thm]{Remark}

\makeatletter \@addtoreset{equation}{section} \makeatother

\newcommand{\R}{\mathbb{R}}
\newcommand{\B}{\mathbb{B}}
\newcommand{\E}{\mathbb{E}}
\newcommand{\M}{\mathbb{M}}

\newcommand{\N}{\mathbb{N}}

\begin{document}
\title{Existence of multiple solutions for quasi-linear degenerate elliptic equations}
\author{Yawei Wei\footnote{Yawei Wei, School of Mathematical Sciences and LPMC, Nankai University, Tianjin 300071, China; e-mail: weiyawei@nankai.edu.cn}}

\pagestyle{headings}
\date{}
\maketitle

\begin{abstract}
The present paper is concerned a class of quasi-linear elliptic degenerate equations. The degenerate operator comes from the analysis of manifolds with corner singularity. Variational methods are applied to verify the existence of infinity many solutions for the problems.
\footnote{Acknowledgements: This article has been supported by the NSFC (National Science Foundation of China) under the grands 11771218, 11371282, and supported by the Fundamental Research Funds for the Central Universities.}
\end{abstract}

\addcontentsline{toc}{section}{Introduction}
%
%
\section{Introduction}

In this paper, the following quasi-linear degenerate elliptic equations is concerned
\begin{equation}\label{1-1}
\bigg\{\begin{array}{ll}
-(x_1 x_2)^{-p} \textup{div}_{\M}(|\nabla_{\M}u|^{p-2}\nabla_{\M}u)=\lambda|u|^{q-2}u,&
\textup{in}\,\, \textup{int}\mathbb{M}\\
u=0& \textup{on}\,\,\partial \mathbb{M}.\end{array}
\end{equation}
where $\lambda >0$, $2<p<N$ and $p\leq q <p^*=\frac{Np}{N-p}$. A local model of stretched manifold with corner singularity is denoted by $\M:=(0,\delta)\times (0,\delta)\times X$, with the fixed small positive $\delta$ and dimension $N=n+2,$ and $\partial \M:= \{0\} \times \{0\}\times X$ denotes the boundary of $\M$, where $X$ is a bounded open set in the unit sphere of $\R^{N-2}$
with $x^\prime:=(x_1^\prime,...,x_n^\prime)\in X$,  $\nabla_\M:=(x_1\partial_{x_1},x_1x_2\partial_{x_2},\partial_{x_1^\prime},...,\partial_{x^\prime_n},)$, and $\mbox{div}_{\M}:=\nabla_{\M}\cdot$.

The non-trivial solutions $u \in \mathcal{H}^{1,(\frac{N-1}{p},\frac{N}{p})}_{p,0}(\M)$ (see the Definition \ref{defM}) verifies \eqref{1-1}
in the weak sense, i.e., for any $\varphi \in C_0^\infty(\mbox{int}\M)$, it holds that
\begin{equation}\label{weakdef}
\int_\M x_1|\nabla_\M u|^{p-2} \nabla_\M u\cdot \nabla_\M\varphi \frac{dx_1}{x_1}\frac{dx_2}{x_1x_2}dx^\prime
  =\lambda \int_\M x_1 (x_1x_2)^p|u|^{q-2} u \varphi\frac{dx_1}{x_1}\frac{dx_2}{x_1x_2}dx^\prime
\end{equation}

In the following calculus, for simplicity, denote $d\sigma:= \frac{dx_1}{x_1}\frac{dx_2}{x_1x_2}dx^\prime$. The weak solutions for \eqref{1-1} are the
critical points of the energy functional
\[
  J(u)=\frac{1}{p}\int_\M x_1|\nabla_\M u|^p d\sigma -\frac{\lambda}{q}\int_\M x_1(x_1x_2)^p |u|^q d\sigma.
\]

The present paper holds the following results.
\begin{Thm}\label{T2}
For $2<p<N$, $p<q <p^*$, and $\lambda>0$ the Dirichlet problem \eqref{1-1} processes infinitely many non-trivial weak solutions in the sense of \eqref{weakdef}.
\end{Thm}
\begin{Thm}\label{geqC1}
If $\{c_m\}_{m\in \N}$ is the critical value sequence obtained in Theorem \ref{T2}, then we have
$c_m \to \infty \quad \textup{as} \quad m \to \infty.$
\end{Thm}
The problem \eqref{1-1} with $p<q<p^*$ holding different homogeneity of the right hand side preserves a curve of solution. In fact, if $u\neq 0$, is a solution of the problem \eqref{1-1} with $\lambda=1$, then for any $\alpha>0$, $\alpha u$ verifies the problem \eqref{1-1} with $\lambda=\alpha^{p-q}$. But for the case
of $p=q$, if $(u,\lambda)$ is a solution of \eqref{1-1}, then for all $\alpha\in \R$, $(\alpha u,\lambda)$ is a solution too. Hence, we need different methods to solve the problem in the two cases. We call the problem \eqref{1-1} with $p=q$ the typical Dirichlet eigenvalue problem, which holds the following results.
\begin{Thm}\label{T1}
For $2<p<N$ and $q = p$, the Dirichlet eigenvalue problem \eqref{1-1} processes a sequence of infinitely many non-trivial weak solutions $(u_k,\lambda_k)\in \mathcal{H}^{1,(\frac{N-1}{p},\frac{N}{p})}_{p,0}(\M)\times \R_+$
in the sense of \eqref{weakdef}.
\end{Thm}

\begin{Thm}\label{C1}
The eigenvalues $\lambda_k$ of \eqref{1-1} in Theorem \ref{T1} turns to infinity as $k \to \infty$.
\end{Thm}
The classical  $p$-Laplacian have been widely studied, such as
\cite{Car}, \cite{Dra}, \cite{Garc} and references therein.
The quasi-linear degenerate operator in \eqref{1-1} comes from the analysis of domain with corner singularities. This academic field has been discussed from various perspectives such as V.Maz'ya \cite{Kozlov}, P.Grisvard \cite{Grisvard}, M.Dauge \cite{Dauge} and R. Melrose \cite{Melr1}. This paper is based on the framework by B.-W. Schulze \cite{Schu20}, and organized as follows. The preliminaries are given in section 2, including definitions and properties of weighted Sobolev spaces, etc. In section 3, some abstract variational methods is applied to verify the problem \eqref{1-1} in the case of $p<q<p^*$. The idea of Lusternik-Schnirelman theory is employed to prove the case of $p=q$ of \eqref{1-1} in section 4.

\section{Preliminaries}

Let $X$ be a bounded open subset in the unit sphere of $\R^n$. Define an infinite cone in $\R^{n+1}$ as a quotient space
$X^\Delta=(\overline{\R}_+\times X)/(\{0\}\times X)$, and the stretched cone is defined as
$X^\wedge=\R_+\times X.$
Set $x_1\in \R_+$, $x^\prime=(x_1^\prime, ..., x_n^\prime)\in X$.
It is sufficient to consider the case of $x_1$ near to $0$, which gives us a finite cone $E=([0,\delta)\times X)/(\{0\}\times X)$ with a small fixed $\delta$. The finite stretched cone to $E$ is $\E=(0,\delta)\times X$,
with the boundary $\partial \E=\{0\}\times X$.
Then an infinite corner can be defined as
$E^\Delta=( \overline{\R}_+\times E)/(\{0\}\times E)$,
and the stretched corner is
$E^\wedge=\R_+\times \E$.
Let $(x_1,x_2,x^\prime)\in E^\wedge$, we focus on the case of $x_2$ small enough, then the finite corner is
$M=([0,\delta)\times E)/(\{0\}\times E)$ and
$\M=(0,\delta)\times \E=(0,\delta)\times (0,\delta)\times X$ denotes a finite stretched corner with the boundary
$\partial \M=\{0\} \times \partial\E=\{0\}\times \{0\}\times X$.

\begin{Def}\label{WLP}
 Let $(x_1,x_2,x^\prime)\in \R_+\times X^\wedge$, with the weight datas $\gamma_1\in \R$, $\gamma_2\in \R$
and $1\leq p<+\infty$.  Then $L_p^{\gamma_1,\gamma_2}(\R_+\times X^\wedge)$
denotes the space of all $u(x)\in \mathcal{D}^\prime(\R_+\times X^\wedge)$ such that
\[\|u\|_{L_p^{\gamma_1,\gamma_2}(\R_+\times X^\wedge)}=\big(\int_{\R_+\times X^\wedge} |x_1^{\frac{N}{p}-\gamma_1}x_2^{\frac{N}{p}-\gamma_2} u(x)|^p d\sigma\big)^{1/p}<+\infty.\](Here and after we denote $d\sigma:= \frac{dx_1}{x_1}\frac{dx_2}{x_1x_2}dx^\prime$ for simplicity.)
The weighted Sobolev spaces are defined as follows
\[
\mathcal{H}^{m,(\gamma_1,\gamma_2)}_p(\R_+\times X^\wedge):=\{u\in\mathcal{D}^\prime(\R_+\times X^\wedge):
 (x_1\partial_{x_1})^l (x_1x_2\partial_{x_2})^j \partial_{x^\prime}^\beta u \in L_p^{\gamma_1,\gamma_2}(\R_+\times X^\wedge)\},
\]
for arbitrary $j ,l \in \N$, $\beta\in \N^{N-2}$, and $j+l+|\beta|\leq m$.
\end{Def}

\begin{Def}\label{defM}
Let $W^{m,p}_{loc}(\textup{int}\M)$ denote the classical local Sobolev space (here $\textup{int}\M$ is interior of $\M$).
For $1\leq p <\infty$, $m\in \N$ and the weighted data $\gamma_1 \in\R$, $\gamma_2 \in\R$
then $\mathcal{H}_{p}^{m,(\gamma_1,\gamma_2)}(\mathbb{M})$
denotes the subspace of all $u\in
W^{m,p}_{\textup{loc}}(\textup{int}\mathbb{M})$, such that
\[
\mathcal{H}_{p}^{m,(\gamma_1,\gamma_2)}(\mathbb{M})=\{u\in
W^{m,p}_{\textup{loc}}(\textup{int}\mathbb{M})~|~ (\omega\sigma)
u\in\mathcal{H}_{p}^{m,(\gamma_1,\gamma_2)}(\R_+\times X^\wedge)\}
\]
for any cut-off functions $\omega = \omega(x_1,x^\prime)$ and $\sigma = \sigma(x_2,x^\prime)$,
supported by a collar neighborhoods of $(0,1)\times \partial\M$ and $(0,1)\times\partial\M$ respectively.
Moreover, define $L_p^{\gamma_1,\gamma_2}(\mathbb{M}):=\mathcal{H}_{p}^{0,(\gamma_1,\gamma_2)}(\mathbb{M}).$
\end{Def}
\begin{Rem}
Although the definitions of weighted Sobolev spaces on manifolds with corner singularity are complex (see more in \cite{SchuWei}), Definitions \ref{WLP} and \ref{defM} fit the present problem \eqref{1-1}. Here since this paper concentrates on $\M=(0,\delta)\times (0,\delta)\times X$ with small enough positive $\delta$, it sufficient to consider the case in the support of $\omega$ and $\sigma$ in the definition \ref{defM}.  Moreover, let $\mathcal{H}^{m,(\gamma_1,\gamma_2)}_{p,0}(\M)$ denote the closure of $C_0^\infty$ in $\mathcal{H}^{m,(\gamma_1,\gamma_2)}_p(\M)$.

\end{Rem}

\begin{Prop}\label{p-ineq}
Let $1\leq p< \infty$ and $\gamma_1,\gamma_2\in \R$. If
$u(x)\in \mathcal{H}^{1,(\gamma_1,\gamma_2)}_{p,0}(\mathbb{M})$, then
\begin{equation}\label{p}
\| u(x)\|_{L_p^{\gamma_1,\gamma_2}(\mathbb{M})}\leq c \|\nabla_\M
u(t,x)\|_{L_p^{\gamma_1,\gamma_2}(\mathbb{M})},
\end{equation}
where the constant $c$ depends only on $\mathbb{M}$ and $p$.
\end{Prop}
\begin{proof}
Follow the same process of Proposition 3.2 in \cite{CLW4}.
\end{proof}
\begin{Rem}
The proposition \ref{p-ineq} implies that the norm $\|u\|_{\mathcal{H}^{1,(\gamma_1,\gamma_2)}_{p,0}(\M)}$ is equivalent to the norm $\|\nabla_{\M} u\|_{L^{\gamma_1,\gamma_2}_p(\M)}$.
\end{Rem}

Next we introduce
some concepts in variational methods in the following. Let $E$ be Banach space.
\begin{Def}\label{ps}
The functional $I$ satisfies the $($PS$){}_c$ condition, if for any
sequence $\{u_k\}\subset E$ with the properties:
$$I(u_k)\to c\quad \textup{and}\quad \parallel I^\prime(u_k)\parallel_{E^\prime}\to
0,$$ there exists a subsequence which is convergent, where
$I^\prime(\cdot)$ is the Fr\'echet differentiation of $I$ and
$E^\prime$ is the dual space of $E$. If it holds for any $c\in
\R$, we say that $I$ satisfies $($PS$)$ condition.
\end{Def}

\begin{Def}\label{genus}
Define the class in $E$
\[
  \Sigma(E)=\{ A\subset E~|~ A\,\,\,is\,\,\, closed,\,\,\, and\,\,\, A=-A \}.
\]For $A\in \Sigma(E)$, define the genus of $A$, denoted by $\gamma(A)$, as
\begin{displaymath}
\gamma(A)=\left\{ \begin{array}{l}
0, \quad\textup{if}\,\,\, A=\emptyset\\
\infty, \quad \textup{if}\,\,\,\{m\in \N_+;\exists \,h\in C(A,\R^m\setminus \{0\}), h(-x)=-h(x)\}=\emptyset\\
\inf\{m\in \N_+;\exists \, h\in C(A,\R^m\setminus \{0\}), h(-x)=-h(x)\}
\end{array}\right.
\end{displaymath}
\end{Def}

\begin{Prop}\label{genusp} Let $A, B\in \Sigma(E)$, the genus $\gamma$ possesses the following properties.
\begin{itemize}
\item[(1)] If $\psi\in C(A,B)$ is odd, then $\gamma(A)\leq \gamma(B)$.
\item[(2)] If $\psi\in C(A,B)$ is an odd homeomorphism, then $\gamma(A)=\gamma(B)=\gamma(\psi(A))$.
\item[(3)] If $A \subset B$, then $\gamma(A)\leq \gamma(B)$.
\item[(4)] If $\gamma(B)< \infty$, $\gamma(\overline{A-B}) \geq \gamma(A)-\gamma(B)$.
\item[(5)] $\gamma(A\cup B)\leq \gamma(A)+\gamma(B)$.
\item[(6)] If $S^{n-1}$ is the sphere in $\R^n$, then $\gamma(S^{n-1})=n$.
\item[(7)] If $A$ is compact, then $\gamma(A)<\infty$.
\item[(8)] If $A$ is compact, there exists $\delta>0$ such that for $N_{\delta}(A)=\{x\in X: d(x,A)<\delta\}$ we have
$\gamma(A)=\gamma(N_{\delta}(A))$.
\end{itemize}
\end{Prop}
\begin{proof}
The proof can be found in section 3 of \cite{Rab1}.
\end{proof}
The abstract theory in \cite{Ambr1} will be employed to investigate the existence of solutions for Dirichlet problem \eqref{1-1}. We recall it in the
 following. Let $E$ be an infinite dimensional Banach space over $\R$. Let the functional $I\in C^1(E,\R)$ and $B_r=\{u\in E|\, \|u\|_E\leq r\}$. For convenience,
 set $B:=B_1$. Assume $I$ satisfies $I(0)=0$ and the following five properties,
\begin{itemize}
\item[\textup{($I_1$)}] the functional $I$ satisfies that $I(u)=I(-u)$ for all $u\in E$;
\item[\textup{($I_2$)}] the functional $I$ verifies the Palais-Smale condition;
\item[\textup{($I_3$)}] there exists a $\rho >0$ such that $I>0$ in $B_\rho\setminus \{0\}$ and $I\geq \alpha>0$ on $\partial B_\rho$;
\item[\textup{($I_4$)}] there exists $v\in E$ such that $\|v\|_E >\rho$ and $I(v)<\alpha$;
\item[\textup{($I_5$)}] for any finite dimensional subspaces $E_m \subset E$, it holds $E_m\cap A_0$ is bounded, where $A_0=\{u\in E|\, 0\leq I(u) <+\infty\}$.
\end{itemize}
Let
$\Gamma:=\{h\in C(E,E)|\,\, h(0)=0;\,\, h \,\,\textup{is odd homeomorphism};\,\,  h(B)\subset A_0\}$ and
$\Gamma_m=\{K\subset E| \,\,K \mbox{ compact};\,\, K=-K;\,\, \gamma(K\cap h(\partial B))\geq m,\,\, \forall h\in \Gamma\}$

\begin{Lem}\label{AR1}
Suppose $I$ satisfies $(I_1)$-$(I_5)$. For each $m\in N$, Let
\begin{equation}\label{eq:bm}
 b_m=\inf_{K\in \Gamma_m}\max_{u\in K} I(u).
\end{equation}
Then $0<\alpha\leq b_m\leq b_{m+1}$ and $b_m$ is a critical value of $I$. Moreover, if $b_{m+1}=\cdots=b_{m+r}=b$, then $\gamma(K_b)\geq r$,
where $K_b=\{u\in E| I^\prime (u)=0, I(u)=b\}$.
\end{Lem}
\begin{proof}
See Theorem 2.8 in \cite{Ambr1}.
\end{proof}
Let $\{E_m\}_{m\in N}$ be a sequence of subspaces of $E$, such that $\dim(E_m)=m$; $E_m \subset E_{m+1}$; $\mathcal{L}(\cup_{m\in \N} E_m)$ denotes the linear manifold generated by $\cup_{m\in \N} E_m$ which is dense in $E$. By $E_m^c$ we denote the algebraically and topologically complementary of $E_m$.
\begin{Lem}\label{AR2}
Let $I$ satisfies $(I_1)$-$(I_5)$. For each $m\in N$, let
\begin{equation}\label{eq:cm}
c_m=\sup_{h\in \Gamma}\inf_{u\in \partial B\cap E_{m-1}^c} I(h(u)).
\end{equation}
Then $0<\alpha\leq c_m\leq b_m\leq \infty$, $c_m\leq c_{m+1}$, and $c_m$ is a critical value of $I$.
\end{Lem}
\begin{proof}
See Theorem 2.13 in \cite{Ambr1}.
\end{proof}
\section{The case of $p<q<p^*$}
\subsection{The proof of Theorem \ref{T2}}

The idea of the proof here is to verify the condition $I_1$-$I_5$ in Lemma \ref{AR1} and Lemma \ref{AR2}. The following lemmas will be applied in the proof.

\begin{Lem}\label{compact}
For $1<p<N$ and $1\leq q <p^*=\frac{Np}{N-p}$ the embedding
\[
  \mathcal{H}_{p,0}^{1,(\gamma_1,\gamma_2)}(\mathbb{M})\hookrightarrow L_q^{\gamma_1^\prime, \gamma_2^\prime}(\M)
\]
if $\frac{N}{q}-\gamma_1^\prime>\frac{N}{p}-\gamma_1$ and $\frac{N}{q}-\gamma_2^\prime>\frac{N}{p}-\gamma_2$
\end{Lem}
\begin{proof}

Since the embedding $\mathcal{H}_{p,0}^{0,(\gamma^\prime_1,\gamma^\prime_2)}(\mathbb{M})\hookrightarrow L_q^{\gamma_1^\prime, \gamma_2^\prime}(\M)$ is continuous, it is sufficient to prove
\[
  [\omega][\sigma]\mathcal{H}_{p,0}^{1,(\gamma_1,\gamma_2)}(\R_+\times\R_+\times X)\hookrightarrow  [\omega][\sigma]\mathcal{H}_{q,0}^{0,(\gamma^\prime_1,\gamma^\prime_2)}(\R_+\times\R_+\times X)
\]
is compact.
Set $1\leq l <\infty$, for any $v(x)\in \mathcal{H}_{l,0}^{m,(\gamma_1,\gamma_2)}(\R_+\times\R_+\times X)$, define
\begin{equation}\label{eq:Slr2}
(\hat{S}_{l\gamma_2}v)(x_1,y,x^\prime) = e^{-y(\frac{N}{l}-\gamma_2)}v(x_1,e^{-y},x^\prime):=w(x_1,y,x^\prime).
\end{equation}
Then $\hat{S}_{l\gamma_2}$ induces an isomorphism
\begin{equation}\label{S1}
\hat{S}_{l\gamma_2}:[\omega][\sigma]\mathcal{H}_{l,0}^{m,(\gamma_1,\gamma_2)}(\R_+\times\R_+\times X)\to [\omega][\tilde{\sigma}]\mathcal{H}_{l,0}^{m,\gamma_1}(\R_+\times\R \times X)
\end{equation}
where $\tilde{\sigma}(y)=\sigma(e^{-y})$ and the $\mathcal{H}^{m,\gamma}_{l,0}(\R_+\times\R\times X)$ (see more in \cite{Schu20} and \cite{SchuWei}) denotes the space of all
 $w(x_1,y,x^\prime)\in \mathcal{D}^\prime(\R_+\times\R\times X)$ such that, for $k,j\in \N$
and $\alpha\in \N^{N-2}$
\begin{equation}\label{wedgesobo}
\|w\|^l_{\mathcal{H}^{m,\gamma}_{l,0}(\R_+\times\R\times X)}=\sum_{k+j+|\alpha|\leq m}\int_{\R_+\times \R\times X}|x_1^{\frac{N}{l}-\gamma_1}(x_1\partial_{x_1})^k(x_1\partial_y)^j\partial_{x^\prime}^\alpha w|^l \frac{dx_1}{x_1}\frac{dy}{x_1}dx^\prime <\infty.
\end{equation}
In fact, we have
\begin{align*}
&\|(\hat{S}_{l\gamma_2}v)(x_1,y,x^\prime)\|^l_{\mathcal{H}_{l,0}^{m,\gamma_1}(\R_+\times\R \times X)}= \|w\|^l_{\mathcal{H}^{m,\gamma}_{l,0}(\R_+\times\R\times X)}\\
=&\sum_{k+j+|\alpha|\leq m}\int_{\R_+\times \R\times X}|x_1^{\frac{N}{l}-\gamma_1}(x_1\partial_{x_1})^k(x_1\partial_y)^j\partial_{x^\prime}^\alpha e^{-y(\frac{N}{p}-\gamma_2)}v(x_1,e^{-y},x^\prime)|^l \frac{dx_1}{x_1}\frac{dy}{x_1}dx^\prime\\
=&c_{(N,p,\gamma_2)}\sum_{k+j+|\alpha|\leq m}\int_{\R_+\times \R\times X}|x_1^{\frac{N}{l}-\gamma_1} x_2^{\frac{N}{p}-\gamma_2} (x_1\partial_{x_1})^k (x_1x_2\partial_{x_2})^j \partial_{x^\prime}^\alpha v(x_1,x_2,x^\prime)|^l \frac{dx_1}{x_1}\frac{dx_2}{x_1x_2}dx^\prime\\
=&\|v(x)\|^l_{\mathcal{H}_{l,0}^{m,(\gamma_1,\gamma_2)}(\R_+\times\R_+\times X)}<\infty
\end{align*}
It proves the isomorphism of $\hat{S}_{l\gamma_2}$ in \eqref{S1}.
Moreover, we need the following map to deal with the degeneracy caused by $x_1$,
\begin{equation}\label{TSr1}
(\tilde{S}_{l\gamma_1}w)(\rho,\xi,x^\prime)=e^{-\rho(\frac{N}{l}-\gamma_1)}w(e^{-\rho},e^{-\rho}\xi,x^\prime):=a(\rho,\xi,x^\prime)
\end{equation}
which induces an isomorphism
\begin{equation}\label{S2}
\tilde{S}_{l\gamma_1}: [\omega][\tilde{\sigma}]\mathcal{H}^{m,\gamma_1}_{l,0}(\R_+\times \R \times X)\to [\tilde{\omega}][\hat{\sigma}]W_0^{m,l}(\R \times\R \times X)
\end{equation}
where $\tilde{\omega}(\rho)=\omega(e^{-\rho})$, and $\hat{\sigma}(\xi)$ is a cut-off function in $\xi$ for $\xi=\frac{y}{x_1}$ with $y\in \textup{supp}\tilde{\sigma}(y)$ and $x_1\in \textup{supp}\omega(x_1)$, and $W^{m,p}(\cdot)$ denotes the classical Sobolev spaces.
In fact, the rule of changing variables implies that
\begin{align*}
&\|\tilde{S}_{l\gamma_1}w\|^l_{W_0^{m,l}(\R \times\R \times X)}
=\sum_{k+j+|\alpha|\leq m}\int_{\R\times \R\times X}|\partial_\rho^k \partial_\xi^j \partial_{x^\prime}^\alpha a(\rho,\xi,x^\prime)|^l d\rho d\xi dx^\prime\\
=&\sum_{k+j+|\alpha|\leq m}\int_{\R\times \R\times X}|\partial_\rho^k \partial_\xi^j \partial_{x^\prime}^\alpha e^{-\rho(\frac{N}{l}-\gamma_1)}w(e^{-\rho},e^{-\rho}\xi,x^\prime)|^l d\rho d\xi dx^\prime\\
=&c_{(N,l,\gamma_1)}\sum_{k+j+|\alpha|\leq m}\int_{\R_+\times \R\times X}|x_1^{\frac{N}{l}-\gamma_1}(x_1\partial_{x_1})^k (x_1\partial_y)^j \partial_{x^\prime}^\alpha w(x_1,y,x^\prime)|^l \frac{dx_1}{x_1}\frac{dy}{x_1} dx^\prime\\
=&c_{(N,l,\gamma_1)}\|w(x_1,y,x^\prime)\|^l_{\mathcal{H}^{m,\gamma_1}_{l,0}(\R_+\times \R \times X)}<\infty
\end{align*}
This induces the isomorphism of $\tilde{S}_{l\gamma_1}$ in \eqref{S2}.
Then set $S_{l(\gamma_1,\gamma_2)}=\tilde{S}_{l,\gamma_1}\circ \hat{S}_{l,\gamma_2}$, for $v(x)\in \mathcal{H}^{m,(\gamma_1,\gamma_2)}_{l,0}(\R_+\times \R_+ \times X)$, we have
\[
  (S_{l,(\gamma_1,\gamma_2)}v)(\rho,\xi,x^\prime)=e^{-\rho(\frac{N}{l}-\gamma_1)}
  e^{-\xi e^{-\rho}(\frac{N}{l}-\gamma_2)}v(e^{-\rho},e^{-\xi e^{-\rho}},x^\prime)
\] which induces the following isomorphism,
\begin{equation}\label{mapS}
  S_{l(\gamma_1,\gamma_2)}=\tilde{S}_{l,\gamma_1}\circ \hat{S}_{l,\gamma_2}:[\omega][\sigma]\mathcal{H}^{m,(\gamma_1,\gamma_2)}_{l,0}(\R_+\times \R_+ \times X)
  \to [\tilde{\omega}][\hat{\sigma}]W_0^{m,l}(\R\times \R \times X)
\end{equation}
Now for $u_q\in\mathcal{H}^{0,(\gamma^\prime_1,\gamma^\prime_2)}_{q,0}(\R_+\times \R_+ \times X)$, we have
\[
  (S_{q,(\gamma_1^\prime,\gamma_2^\prime)}[\omega][\sigma]u_q)(\rho,\xi,x^\prime)=[\tilde{\omega}][\hat{\sigma}]
  e^{-\rho(\frac{N}{q}-\gamma^\prime_1)}e^{-\xi e^{-\rho}(\frac{N}{q}-\gamma^\prime_2)}u_q(e^{-\rho},e^{-\xi e^{-\rho}},x^\prime)
\]
which gives the following isomorphism
\[
  S_{q,(\gamma_1^\prime,\gamma_2^\prime)}:[\omega][\sigma]\mathcal{H}^{0,(\gamma^\prime_1,\gamma^\prime_2)}_{q,0}(\R_+\times \R_+ \times X)\to [\tilde{\omega}][\hat{\sigma}]W^{0,q}_0(\R\times\R \times X).
\]
In the other hand, the map $S_{q,(\gamma_1^\prime,\gamma_2^\prime)}$ induces another isomorphism, for $u_p\in \mathcal{H}^{1,(\gamma_1,\gamma_2)}_{p,0}(\R_+\times \R_+ \times X)$, as follows. Set $\delta_1:=(\frac{N}{q}-\gamma^\prime_1)-(\frac{N}{p}-\gamma_1)$, $\delta_2:=(\frac{N}{q}-\gamma^\prime_2)-(\frac{N}{p}-\gamma_2)$, then we have
\begin{align*}
&(S_{q,(\gamma_1^\prime,\gamma_2^\prime)}[\omega][\sigma]u_p)(\rho,\xi,x^\prime)
=[\tilde{\omega}][\hat{\sigma}]e^{-\rho(\frac{N}{q}-\gamma^\prime_1)}e^{-\xi e^{-\rho}(\frac{N}{q}-\gamma^\prime_2)}
u_p(e^{-\rho},e^{-\xi e^{-\rho}},x^\prime)\\
=&[\tilde{\omega}][\hat{\sigma}]e^{-\rho\delta_1}e^{-\xi e^{-\rho}\delta_2}
e^{-\rho(\frac{N}{p}-\gamma_1)}e^{-\xi e^{-\rho}(\frac{N}{p}-\gamma_2)}u_p(e^{-\rho},e^{-\xi e^{-\rho}},x^\prime)\\
\end{align*}
which gives the isomorphism
\[
  S_{q,(\gamma_1^\prime,\gamma_2^\prime)}:[\omega][\sigma]\mathcal{H}^{1,(\gamma_1,\gamma_2)}_{p,0}(\R_+\times \R_+ \times X)\to
   [\tilde{\omega}][\hat{\sigma}]e^{-\rho\delta_1}e^{-\xi e^{-\rho}\delta_2}W^{1,p}_0(\R\times \R\times X).
\]
For $1< q< p^*$ and $\delta_1>0$, $\delta_2>0$, the following embedding is compact
\[
  [\tilde{\omega}][\hat{\sigma}]e^{-\rho\delta_1}e^{-\xi e^{-\rho}\delta_2}W^{1,p}_0(\R\times \R\times X)\hookrightarrow   [\tilde{\omega}][\hat{\sigma}]W^{0,q}_0(\R\times\R \times X)
\]
since the functions $e^{-\rho\delta_1}$ and $e^{-\xi e^{-\rho}\delta_2}$ vanish rapidly as $\rho\to \infty$ and $\xi\to \infty$, then
the function $\varphi(\rho\xi)=e^{-\rho\delta_1}\rho^{s_1}e^{-\xi e^{-\rho}\delta_2}\xi^{s_2}$ and all the derivatives in $\rho$ and $\xi$
are uniformly bounded on $\textup{supp}\tilde{\omega}$ and $\textup{supp}\hat{\sigma}$ for every $s_1, s_2\in \R$.
\end{proof}
\begin{Rem}\label{embconti}
Apply the same idea in Lemma \ref{compact}, for $1<p<N$ and $1\leq q< p^*$ the embedding
\[
  \mathcal{H}_{p,0}^{1,(\gamma_1,\gamma_2)}(\mathbb{M})\hookrightarrow L_q^{\gamma_1^\prime, \gamma_2^\prime}(\M)
\]
is continuous, if $\frac{N}{q}-\gamma_1^\prime \geq \frac{N}{p}-\gamma_1$ and $\frac{N}{q}-\gamma_2^\prime \geq \frac{N}{p}-\gamma_2$.
The embedding
\[
  \mathcal{H}_{p,0}^{m^\prime,(\gamma^\prime_1,\gamma^\prime_2)}(\mathbb{M})\hookrightarrow  \mathcal{H}_{p,0}^{m,(\gamma_1,\gamma_2)}(\mathbb{M})
\]
is continuous if $m^\prime\geq m$, $\gamma_1^\prime\geq \gamma_1$, and $\gamma_2^\prime\geq \gamma_2$.
\end{Rem}

\begin{Lem}[Breizis-Lieb type result]\label{BLR}
Let $1\leq p <\infty$ and $\{u_k\}\subset L_p^{\gamma_1,\gamma_2}(\M)$. If the following
conditions are satisfied
\begin{itemize}
\item[(i)] $\{u_k\}$ is bounded in $L_p^{\gamma_1,\gamma_2}(\M)$,
\item[(ii)] $u_k \to u$ a.e in $\textup{int}\M$, as $k\to \infty$,
\end{itemize} then
\begin{equation}\label{eq:BLR}
\lim_{k\to \infty}(\|u_k\|^p_{L_p^{\gamma_1,\gamma_2}(\M)}-\|u_k-u\|^p_{L_p^{\gamma_1,\gamma_2}(\M)})=\|u\|^p_{L_p^{\gamma_1,\gamma_2}(\M)}
\end{equation}
\end{Lem}
\begin{proof}
Due to Fatou Lemma, it yields
\begin{align*}
 \|u\|^p_{L_p^{\gamma_1,\gamma_2}}&=\int_\M |x_1^{\frac{N}{p}-\gamma_1}x_2^{\frac{N}{p}-\gamma_2} u|^p d\sigma\\
 &\leq \liminf_{k\to \infty} \int_\M |x_1^{\frac{N}{p}-\gamma_1}x_2^{\frac{N}{p}-\gamma_2} u_k|^p d\sigma
 =\liminf_{k\to \infty} \|u_k\|^p_{L_p^{\gamma_1,\gamma_2}}<\infty
\end{align*}
For simplicity, we set here $\tilde{u}_k=x_1^{\frac{N}{p}-\gamma_1}x_2^{\frac{N}{p}-\gamma_2} u_k$ and
$\tilde{u}=x_1^{\frac{N}{p}-\gamma_1}x_2^{\frac{N}{p}-\gamma_2} u$.
Since $p>1$, then $j(t)=t^p$ is convex. For any fixed $\varepsilon >0$, there exists a constant $c_{\varepsilon}$,
such that
\[
  \big||\tilde{u}_k-\tilde{u}+\tilde{u}|^p+|\tilde{u_k}-\tilde{u}|^p\big|\leq \varepsilon |\tilde{u}_k-\tilde{u}|^p+c_\varepsilon |\tilde{u}|^p,
\] and then
\[
   \big||\tilde{u}_k-\tilde{u}+\tilde{u}|^p-|\tilde{u_k}-\tilde{u}|^p-|\tilde{u}|^p \big|\leq \varepsilon |\tilde{u}_k-\tilde{u}|^p+(1+c_\varepsilon) |\tilde{u}|^p.
\]
Therefore, we obtain that \[
  f_k^\varepsilon:=( \big||\tilde{u}_k|^p-|\tilde{u_k}-\tilde{u}|^p-|\tilde{u}|^p \big|-\varepsilon |\tilde{u}_k-\tilde{u}|^p)^+\leq (1+c_\varepsilon) |\tilde{u}|^p
\]
Then Lebesgue dominate theorem induces
\[
  \lim_{k\to \infty} \int_\M f_k^\varepsilon (x)d\sigma =\int_\M \lim_{k\to \infty} f_k^\varepsilon (x)d\sigma=0.
\]
Since
\begin{align*}
&\big||x_1^{\frac{N}{p}-\gamma_1}x_2^{\frac{N}{p}-\gamma_2} u_k|^p-|x_1^{\frac{N}{p}-\gamma_1}x_2^{\frac{N}{p}-\gamma_2} u_k-x_1^{\frac{N}{p}-\gamma_1}x_2^{\frac{N}{p}-\gamma_2} u|^p-|x_1^{\frac{N}{p}-\gamma_1}x_2^{\frac{N}{p}-\gamma_2} u|^p\big|\\
  \leq & f_k^\varepsilon+ \varepsilon |x_1^{\frac{N}{p}-\gamma_1}x_2^{\frac{N}{p}-\gamma_2} u_k-x_1^{\frac{N}{p}-\gamma_1}x_2^{\frac{N}{p}-\gamma_2} u|^p,
\end{align*}
then for any arbitrary small $\varepsilon$, it follows that
\[
  \limsup_{k\to \infty} \int_\M \big||x_1^{\frac{N}{p}-\gamma_1}x_2^{\frac{N}{p}-\gamma_2} u_k|^p-|x_1^{\frac{N}{p}-\gamma_1}x_2^{\frac{N}{p}-\gamma_2} (u_k- u)|^p-|x_1^{\frac{N}{p}-\gamma_1}x_2^{\frac{N}{p}-\gamma_2} u|^p\big|d\sigma \leq c\cdot \varepsilon
\]
 where $c:=\sup \int_\M  |x_1^{\frac{N}{p}-\gamma_1}x_2^{\frac{N}{p}-\gamma_2}( u_k- u)|^p d\sigma.$ It verifies the result.
\end{proof}
By a direct calculation, one can derive that the energy functional
 \[
  J(u)=\frac{1}{p}\int_\M x_1|\nabla_\M u|^p d\sigma -\frac{\lambda}{q}\int_\M x_1(x_1x_2)^p |u|^q d\sigma \in C^1(\mathcal{H}^{1,(\frac{N-1}{p},\frac{N}{p})}_{p,0}(\M),\R)
\] satisfies $J(0)=0$ and $J(u)=J(-u)$ for any $u\in \mathcal{H}^{1,(\frac{N-1}{p},\frac{N}{p})}_{p,0}(\M)$.
\begin{Lem}\label{PScond}
Let $p<q<p^*$, then the functional
\[
   J(u)=\frac{1}{p}\int_\M x_1|\nabla_\M u|^p d\sigma -\frac{\lambda}{q}\int_\M x_1(x_1x_2)^p |u|^q d\sigma
\]
verifies the $($PS$)$ condition.
\end{Lem}
\begin{proof}
Let $\{u_k(x)\}\in \mathcal{H}^{1,\frac{N-1}{p}}_{p,0}(\M)$ be a \textup(PS) sequence. Then
 \[
   J(u_k)-\frac{1}{q}<J^\prime(u_k),u_k>
=(\frac{1}{p}-\frac{1}{q})\int_{\M} x_1|\nabla_{\M} u|^p d\sigma < \infty
 \]
 which implies that
$
 \{\|u_k\|_{\mathcal{H}^{1,(\frac{N-1}{p},\frac{N}{p})}_{p,0}(\M)}\}
$ is bounded. Hence
\[
  u_k \rightharpoonup u  \,\,\, \textup{in}\,\,\, \mathcal{H}^{1,(\frac{N-1}{p},\frac{N}{p})}_{p,0}(\M),\mbox{ as  } k\to \infty,
\] and together with Lemma \ref{compact}, it follows
\[
  u_k  \to u \,\,\, \textup{in}\,\,\, L_q^{\gamma_1,\gamma_2}(\M),\mbox{ as  } k\to \infty,
\] for $1<q<p^*$ and $\frac{1}{p}<\frac{N}{q}-\gamma_1<p+1$, $0<\frac{N}{q}-\gamma_2<p$.
Let us calculate that
\begin{align*}
o(1)=&<J^\prime(u_k)-J^\prime(u),u_k-u>\\
=&\int_{\M}(|\nabla_{\M} x_1 u_k|^{p-2}\nabla_{\M}u_k-|\nabla_{\M}u|^{p-2}\nabla_{\M}u)(\nabla_{\M}u_k-\nabla_{\M}u)d\sigma\\
-&\lambda\int_{\M}x_1^{p+1}x_2^p(|u_k|^{q-2}u_k-|u|^{q-2}u)(u_k-u)d\sigma
=:I_1-I_2
\end{align*}
Due to H\"older inequality, we derive that $I_2\leq \lambda T_1\cdot T_2$, with
 \[
 T_1:=(\int_{\M}|x_1^{\frac{N}{q}-\gamma_1}x_2^{\frac{N}{q}-\gamma_2}(u_k-u)|^qd\sigma)^{\frac{1}{q}}
\]
\[
  T_2:=(\int_{\M}|x_1^{p+1-(\frac{N}{q}-\gamma_1)}x_2^{p-(\frac{N}{q}-\gamma_2)}(|u_k|^{q-2}u_k-|u|^{q-2}u)|^{\frac{q}{q-1}} d\sigma)^{\frac{q-1}{q}}
\]
Since $\{u_k\}$ is bounded in $\mathcal{H}^{1,(\frac{N-1}{p},\frac{N}{p})}_{p,0}(\M)$ and $u_k\to u$ in $L_q^{\gamma_1,\gamma_2}(\M)$, we derive that
$T_1 \to 0$ and $T_2$ is bounded which implies $I_2 \to 0, \mbox{ as } k\to \infty.$
Then we arrive that
\begin{equation}\label{I1}
I_1=\int_{\M}P_k(x) d\sigma \to 0.
\end{equation}
where $P_k(x)=x_1(|\nabla_{\M}u_k|^{p-2}\nabla_{\M}u_k-|\nabla_{\M}u|^{p-2}\nabla_{\M}u)(x)(\nabla_{\M}u_k-\nabla_{\M}u)(x)$,
Here, denote the $i^{th}$ component of $\nabla_\M u$ by $(\nabla_\M u)_i$.
It is easy to verify that
 $P_k(x)\geq 0; \mbox{ and  } P_k(x)>0, \mbox{ if  } \nabla_{\M} u_k \neq \nabla_{\M} u.$
In the following, we show that
\begin{equation}\label{ith}
(\nabla_\M u_k)_i \to (\nabla_\M u)_i \,\,\,\mbox{ for } 1\leq i\leq N, \mbox{ as } k\to \infty
\end{equation}
 a.e in $\textup{int} \M$, which can be deduced by contradiction. Assume, there exists a point $x_p\in \textup{int} \M$,
and its neighborhood $U_{x_p}$, such that for any $x_0\in U_{x_p}$,
\[
  \lim_{k\to \infty} \nabla_{\M}u_k(x_0)\neq \nabla_{\M}u(x_0).
\]
Since $ x_1(|\nabla_{\M}u_k|^{p-2}\nabla_{\M}u_k-|\nabla_{\M}u|^{p-2}\nabla_{\M}u)_i(x_0)(\nabla_{\M}u_k-\nabla_{\M}u)_i(x_0)\leq c,$
it follows that
\begin{align*}
&x_1(|\nabla_{\M}u_k|^{p-2}\nabla_{\M}u_k)_i(x_0)(\nabla_\M u_k)_i(x_0)\\
\leq & c+x_1(|\nabla_{\M}u_k|^{p-2}+|\nabla_{\M}u|^{p-2})(x_0)(\nabla_\M u_k)_i(x_0)(\nabla_\M u)_i(x_0),
\end{align*}
which indicates that
$\{x_1|\nabla_\M u_k(x_0)|^p\}$ is bounded. There exists a subsequence, here still denoted by $\{u_k\}$ such that
\[
  (\nabla_\M u_k)(x_0)\to \xi^\prime \neq \xi=\nabla_\M u(x_0), \mbox{ as } k\to \infty.
\]
This induces that
\[
  P_k(x_0)=x_1(|\nabla_{\M}u_k|^{p-2}\nabla_{\M}u_k-|\nabla_{\M}u|^{p-2}\nabla_{\M}u)(x_0)(\nabla_{\M}u_k-\nabla_{\M}u)(x_0)\to c_0> 0,
\]
for any $x_0\in U_{x_p}$, as $k\to \infty$. It follows that
\[
  I_1=\int_\M P_k(x)d\sigma \to c\neq 0, \mbox{ as } k\to \infty,
\]
which contradicts to \eqref{I1}, and then \eqref{ith} is obtained.
Applying Lemma \ref{BLR} to $(\nabla_\M u_k)_i$, for $1\leq i\leq N$, we have
\begin{equation}\label{BLRNA}
\lim_{k\to \infty}(\|\nabla_{\M}u_k\|^p_{L_p^{(\frac{N-1}{p},\frac{N}{p})}(\M)}-\|\nabla_{\M}u_k-\nabla_{\M}u\|^p_{L_p^{(\frac{N-1}{p},\frac{N}{p})}(\M)})
=\|\nabla_{\M}u\|^p_{L_p^{(\frac{N-1}{p},\frac{N}{p})}(\M)}
\end{equation}
To the end, what left is to show that
\begin{equation}\label{Lpcon}
\int_\M x_1|\nabla_\M u_k|^p d\sigma \to \int_\M x_1|\nabla_\M u|^p d\sigma, \mbox{ as } k\to \infty.
\end{equation}
Due to Egorov Theorem, we obtain that for any $\delta>0$, there exists a subset $E\subset \textup{int}\M$ with the measure $m(E)<\delta$, such that
\[
  (\nabla_\M u_k)_i \to (\nabla_\M u)_i \,\,\,\mbox{ for } 1\leq i\leq N,  \mbox{ as } k\to \infty,
\] uniformly on $\textup{int}\M\setminus E$. It follows that
\begin{equation}\label{eqBME}
\int_{\M\setminus E} x_1|\nabla_{\M} u_k|^p d\sigma \to \int_{\M\setminus E}x_1|\nabla_{\M} u|^p d\sigma, \mbox{ as } k\to \infty.
\end{equation}
Now we claim that for any $\varepsilon >0$, there is $\delta(\varepsilon)>0$, and a subset $E\subset \M$ with the measure $m(E)<\delta(\varepsilon)$, such that
\begin{equation}\label{eqE}
\int_E x_1|\nabla_{\M} u_k|^p d\sigma<\varepsilon.
\end{equation}
In fact,
\[
  o(1)=I_1
=\int_{\M} x_1(|\nabla_\M u_k|^{p-2}\nabla_\M u_k-|\nabla_\M u|^{p-2}\nabla_\M u)(\nabla_{\M} u_k-\nabla_{\M} u)d\sigma,
\]
which implies that for any $E\subset \M$, we have
\begin{align}\label{Eeq}
\int_E x_1|\nabla_{\M} u_k|^p d\sigma \leq \int_E x_1|\nabla_{\M} u_k|^{p-1}|\nabla_{\M} u|+x_1|\nabla_{\M} u|^{p-1}|\nabla_{\M} u_k|+ x_1|\nabla_{\M} u_k|^{p}d\sigma+o(1).
\end{align}
Applying H\"older inequality on \eqref{Eeq}, it verifies \eqref{eqE}.
Hence, for any $\varepsilon>0$ there exists $\delta(\varepsilon)>0$ and a subset $E\subset \textup{int}\B$, such that both \eqref{eqBME} and \eqref{eqE} hold.
This gives \eqref{Lpcon}
\end{proof}
The following two propositions verifies that the functional $J(u)$ satisfies the conditions $I_3$, $I_4$, $I_5$ in Lemma \ref{AR1} and Lemma \ref{AR2}.
\begin{Prop}
If $p< q < p^*$, then there exists $r>0$ such that
\begin{itemize}
\item[(i)] $J(u)>0$ if $0<\|u\|_{\mathcal{H}^{1,(\frac{N-1}{p},\frac{N}{p})}_{p,0}}<r$ \quad \textup{and} \quad $J(u)\geq \alpha>0$ if $\|u\|_{\mathcal{H}^{1,(\frac{N-1}{p},\frac{N}{p})}_{p,0}}=r$.
\item[(ii)] there exists $v\in \mathcal{H}^{1,(\frac{N-1}{p},\frac{N}{p})}_{p,0}$ such that $\|v\|_{\mathcal{H}^{1,(\frac{N-1}{p},\frac{N}{p})}_{p,0}}>r$ and $J(v)<\alpha$.
\end{itemize}
\end{Prop}
\begin{proof}
According to both Lemma \ref{compact} and the condition $q<p*<p(p+1)$, it holds that
\[
  J(u)\geq \frac{1}{p}\|u\|^p_{\mathcal{H}^{1,(\frac{N-1}{p},\frac{N}{p})}_{p,0}}-\frac{c\lambda}{q}\|u\|^q_{\mathcal{H}^{1,(\frac{N-1}{p},\frac{N}{p})}_{p,0}}
=\|u\|^p_{\mathcal{H}^{1,(\frac{N-1}{p},\frac{N}{p})}_{p,0}}(\frac{1}{p}-\frac{c\lambda}{q}\|u\|^{q-p}_{\mathcal{H}^{1,(\frac{N-1}{p},\frac{N}{p})}_{p,0}})
\]
Let $r=(\frac{q}{2 p c\lambda})^{\frac{1}{q-p}}>0$, if $\|u\|_{\mathcal{H}^{1,(\frac{N-1}{p},\frac{N}{p})}_{p,0}}=r$, then
$J(u)\geq \alpha=\frac{1}{2p}r^p >0$ and if $0<\|u\|_{\mathcal{H}^{1,(\frac{N-1}{p},\frac{N}{p})}_{p,0}}<r$, then $J(u)>\alpha>0$. Then the condition $(i)$ is proved.
Set $\|u\|_{\mathcal{H}^{1,(\frac{N-1}{p},\frac{N}{p})}_{p,0}}=r$, and $\theta>0$, it holds that
$J(\theta u) \to -\infty$ as $\theta \to \infty$.
Therefore, by choosing a large enough positive constant $\theta_1$ such that $v=\theta_1 u$ and $\|v\|_{\mathcal{H}^{1,(\frac{N-1}{p},\frac{N}{p})}_{p,0}}>r$,
one has $J(v)<0<\alpha$, which implies the condition $(ii)$.
\end{proof}
Let $\{E_m\}_{m\in N}$ be a sequence of subspaces of $\mathcal{H}^{1,(\frac{N-1}{p},\frac{N}{p})}_{p,0}(\M)$, such that $\dim(E_m)=m$; $E_m \subset E_{m+1}$; $\mathcal{L}(\cup_{m\in \N} E_m)$ denotes the linear manifold generated by $\cup_{m\in \N} E_m$ which is dense in $E$. By $E_m^c$ we denote the algebraically and topologically complementary of $E_m$.
\begin{Prop}\label{Pbound}
Let $E_m\subset \mathcal{H}^{1,(\frac{N-1}{p},\frac{N}{p})}_{p,0}(\M)$ be defined as above, we have
\[
  P_m=E_m \cap \{ u\in \mathcal{H}^{1,(\frac{N-1}{p},\frac{N}{p})}_{p,0}(\M)| \,0\leq J(u)<+\infty \}
\] is a bounded set.
\end{Prop}
We omit the easy proof of Proposition \ref{Pbound} here for the limit length of writing. Set
\[
  A_0=\{u\in \mathcal{H}^{1,(\frac{N-1}{p},\frac{N}{p})}_{p,0}(\M)~|~ 0\leq J(u) <+\infty\} \quad
     B=\{u\in\mathcal{H}^{1,(\frac{N-1}{p},\frac{N}{p})}_{p,0}(\M)~|~ \|u\|_{\mathcal{H}^{1,(\frac{N-1}{p},\frac{N}{p})}_{p,0}}\leq 1\}.
\]
\[
  \Gamma:=\{h\in C(\mathcal{H}^{1,(\frac{N-1}{p},\frac{N}{p})}_{p,0}(\M),\mathcal{H}^{1,(\frac{N-1}{p},\frac{N}{p})}_{p,0}(\M))~|~ h(0)=0;\, h \,\textup{is odd homeomorphism};\, h(B)\subset A_0\},
\]
\[
\Gamma_m=\{K\subset \mathcal{H}^{1,(\frac{N-1}{p},\frac{N}{p})}_{p,0}(\M)| K \textup{compact}; K=-K; \gamma(K\cap h(\partial B))\geq m, \forall h\in \Gamma\}.
\]

Combining Lemma \ref{AR1} and Lemma \ref{AR2}, it completes the proof of Theorem \ref{T2}.

\subsection{The proof of Theorem \ref{geqC1} }
In this proof, the following definition and lemma will employed.
\begin{Def}\label{geqM}
Define the manifold $M$ as follows
\[
 M=\{u\in \mathcal{H}^{1,(\frac{N-1}{p},\frac{N}{p})}_{p,0}(\M)\setminus \{0\}~|~\, \|u\|^p_{\mathcal{H}^{1,(\frac{N-1}{p},\frac{N}{p})}_{p,0}}=\lambda \int_\M x_1^{p+1}x_2^p |u|^q d\sigma\}.
\]
\end{Def}

\begin{Lem}\label{betau}
For any $u\in \mathcal{H}^{1,(\frac{N-1}{p},\frac{N}{p})}_{p,0}(\M)\setminus \{0\}$, there exists a unique
\[
  \beta:=\beta(u)\geq 0\quad \textup{such that}\quad \beta u\in M.
\]
 The maximum of $J(\beta u)$ for
$\beta \geq 0$ is achieved at $\beta=\beta(u)>0$. The function
$u\mapsto \beta=\beta(u)$ is continuous.
\end{Lem}
\begin{proof}
Let $u\in \mathcal{H}^{1,(\frac{N-1}{p},\frac{N}{p})}_{p,0}(\M)\setminus \{0\}$ be fixed, define
$g(\beta):=J(\beta u)$ on $ [0,\infty).$
Then it follows that
\begin{equation}\label{gprime}
g^\prime(\beta)=0 \Longleftrightarrow \beta u\in M \Longleftrightarrow \| u\|^p_{\mathcal{H}^{1,(\frac{N-1}{p},\frac{N}{p})}_{p,0}}
=\frac{1}{\beta^p}\int_\M x_1^{p+1}x_2^p |\beta u|^q d\sigma.
\end{equation}
It is obvious that $g(0)=0$; $g(\beta)>0$ for $\beta>0$ small enough; and $g(\beta)<0$ for $\beta>0$ large.
Therefore, $\max_{[0,\infty)}g(\beta)$ is achieved at a unique $\beta=\beta(u)$ such that
$g^\prime(\beta)=0$ and $\beta u\in M.$

 To prove the continuity of $\beta(u)$, let us assume that
$u_n \to u$ in $ \mathcal{H}^{1,(\frac{N-1}{p},\frac{N}{p})}_{p,0}(\M)\setminus \{0\}$.
Then $\{\beta(u_n)\}$ is bounded. If a subsequence of $\{\beta(u_n)\}$ converges to $\beta_0$, then it follows from the right side of \eqref{gprime} that
$\beta_0=\beta(u)$.
\end{proof}
By Definition \ref{geqM}, there exists $r>0$, such that
  \begin{equation}\label{geqrinM}
  \int_\M x_1^{p+1}x_2^p|u|^q d\sigma>r, \quad \textup{for any}\,\,\, u\in M.
  \end{equation}
Indeed, if $u\in M$, then by Lemma \ref{compact}, it follows
\begin{equation}\label{eq:4-1}
 \|u\|^p_{\mathcal{H}^{1,(\frac{N-1}{p},\frac{N}{p})}_{p,0}}=\lambda\|u\|^q_{L_q^{(\frac{N-p-1}{q},\frac{N-p}{q})}}\leq c\lambda\|u\|^q_{\mathcal{H}^{1,(\frac{N-1}{p},\frac{N}{p})}_{p,0}}.
\end{equation}
For $q>p$, then it follows $\|u\|^p_{\mathcal{H}^{1,(\frac{N-1}{p},\frac{N}{p})}_{p,0}}\geq (\frac{1}{c\lambda})^{\frac{p}{q-p}}$. Set $r=\frac{1}{2\lambda}(\frac{1}{c\lambda})^{\frac{p}{q-p}}$, it holds that \eqref{geqrinM}.

Let $d_m=\inf\{\|u\|_{\mathcal{H}^{1,(\frac{N-1}{p},\frac{N}{p})}_{p,0}}~|~u\in M\cap E_m^c\}$, then we claim that
\begin{equation}\label{dm}
 d_m\to \infty \quad \textup{as} \quad m\to \infty.
\end{equation}
In fact, if there exist $d>0$ and $u_m\in M\cap E_m^c$, such that
$\|u_m\|_{\mathcal{H}^{1,(\frac{N-1}{p},\frac{N}{p})}_{p,0}}\leq d $  for all $m\in \N_+$.
Then there exists $u\in \mathcal{H}^{1,(\frac{N-1}{p},\frac{N}{p})}_{p,0}(\M)$, such that
$u_m \rightharpoonup u$ in $\mathcal{H}^{1,\frac{N}{p}}_{p,0}(\M)$.
Since $u_m\in E_m^c$, and $\mathcal{L}(\cup_{m\in\N} E_m)$ is dense in $\mathcal{H}^{1,(\frac{N-1}{p},\frac{N}{p})}_{p,0}(\M)$, then we have $u=0$.
According to Lemma \ref{compact}, it follows that
$u_m \to 0$ in $L_q^{(\frac{N-p-1}{q},\frac{N-p}{q})}(\M)$.
This is a contradiction to \eqref{geqrinM}. That means $d_m$ will be unbounded as $m\to\infty$,
which proves the claim \eqref{dm}.

Next, for some $R>1$,  we define a homeomorphism
\begin{equation}\label{hm}
 h_m=R^{-1}d_m u: E_m^c\to E_m^c
\end{equation}
By Lemma \ref{betau}, let $\beta:=\beta(u)$ such that $\beta u\in M$.
Set \[
  B=\{u\in \mathcal{H}^{1,(\frac{N-1}{p},\frac{N}{p})}_{p,0}(\M)~|~ \|u\|_{\mathcal{H}^{1,(\frac{N-1}{p},\frac{N}{p})}_{p,0}}\leq 1\}.
\]For $u_1 \in E_m^c\cap B$, $u_1\neq 0$ and $R>1$, we have
\begin{equation}\label{beta}
 R^{-1}d_m<d_m=\inf\{\|u\|_{\mathcal{H}^{1,(\frac{N-1}{p},\frac{N}{p})}_{p,0}}~|~u\in M\cap E_m^c\}\leq \|\beta u_1\|_{\mathcal{H}^{1,(\frac{N-1}{p},\frac{N}{p})}_{p,0}}\leq \beta:=\beta(u_1).
\end{equation}
It follows that
\begin{equation}\label{hminA}
h_m(E_m^c\cap B)\subset A_0:=\{u\in \mathcal{H}^{1,(\frac{N-1}{p},\frac{N}{p})}_{p,0}(\M)~|~ 0\leq J(u) <+\infty\}.
\end{equation}
In fact, if $u\in E_m^c\cap B$, with $\beta$ chosen as above, such that $\beta u \in M$ and $d_m\leq \beta$, then
\begin{align}\label{4-4}
&J(h_m(u))=\frac{1}{p}(R^{-1}d_m)^p\|u\|^p_{\mathcal{H}^{1,(\frac{N-1}{p},\frac{N}{p})}_{p,0}}-\frac{\lambda}{q}(\frac{R^{-1}d_m}{\beta})^q\|\beta u\|^q_{L_q^{(\frac{N-p-1}{q},\frac{N-p}{q})}}\notag\\
&= \frac{1}{p}(R^{-1}d_m)^p\|u\|^p_{\mathcal{H}^{1,(\frac{N-1}{p},\frac{N}{p})}_{p,0}}-\frac{1}{q}(\frac{R^{-1}d_m}{\beta})^q \|\beta u\|^p_{\mathcal{H}^{1,(\frac{N-1}{p},\frac{N}{p})}_{p,0}}.
\end{align}
Then let $R$ be large enough, it gives $J(h_m(u))\geq 0$, which proves that \eqref{hminA}.

Therefore, we can define
\begin{displaymath}
\widetilde{h_m}(u)=\left\{\begin{array}{ll}
h_m(u) &\textup{if}\,\,\, u\in E_m^c\\
\varepsilon e_j, j=1,2,...,m \,\,\,\textup{and}\,\,\, \{e_j\}_{j=1}^m\,\,\, \textup{is basis of}\,\,\, E_m & \textup{if}\,\,\, u\in E_m
\end{array}\right.
\end{displaymath} for $\varepsilon$ small enough. In this way, it is shown that for $R$ large enough, the mapping $h_m$ in \eqref{hm} defined on $E_m^c$ admits
an extension $\widetilde{h_m}\in \Gamma$ for each $m$.
Finally, we take $u\in \partial B\cap E_m^c$, then
\begin{align}\label{F}
J(\widetilde{h_m}(u))= (R^{-1}d_m)^p\big(\frac{1}{p}-\frac{1}{q}(\frac{R^{-1}d_m}{\beta})^{q-p}\big)\|u\|^p_{\mathcal{H}^{1,(\frac{N-1}{p},\frac{N}{p})}_{p,0}}
\end{align}
where the calculus in \eqref{F} is the same as that in \eqref{4-4}. Since $d_m\leq \beta:=\beta(u)$ proved in \eqref{beta}, then we choose $R$ large enough to deduce that

\[
  J(\widetilde{h_m}(u))\geq \frac{1}{2p}(R^{-1}d_m)^p\to \infty \quad \textup{as}\quad m\to \infty.
\]
Since $\{c_m\}$ is critical value sequence of $J$ (as defined by \eqref{eq:cm}), thus we have
$c_m\to \infty $ as $m \to \infty$. Theorem \ref{geqC1} is proved.

\section{The case of $p=q$}
\subsection{The proof of Theorem \ref{T1}}
The idea of Lusternik-Schnirelman theory in \cite{Amann} is adapted here for the proof. Consider the following two operators,
\begin{equation}\label{BB}
B(u)=\frac{1}{p}\int_{\M}x_1^{p+1}x_2^p |u|^p d\sigma: \mathcal{H}_{p,0}^{1,(\frac{N-1}{p},\frac{N}{p})}(\mathbb{M}) \to \R
\end{equation}
\begin{equation}\label{LB}
b(u)=x_1^{p+1}x_2^p |u|^{p-2}u:\mathcal{H}_{p,0}^{1,(\frac{N-1}{p},\frac{N}{p})}(\mathbb{M}) \to  \mathcal{H}_{p}^{-1,(-\frac{N-1}{p},-\frac{N}{p})}(\mathbb{M})
\end{equation}
where $\mathcal{H}_{p}^{-1,(-\frac{N-1}{p},-\frac{N}{p})}(\mathbb{M})$ is the dual space of $\mathcal{H}_{p,0}^{1,(\frac{N-1}{p},\frac{N}{p})}(\mathbb{M})$ with the norm as follows
\[
  \|g\|_{\mathcal{H}_{p}^{-1,(-\frac{N-1}{p},-\frac{N}{p})}}=\sup_{\varphi}\frac{|<g,\varphi>|}{\|\varphi\|_{\mathcal{H}_{p,0}^{1,(\frac{N-1}{p},\frac{N}{p})}}}
\]
\begin{Lem}\label{Bsprop}
We have the following properties of the above two operators.
\begin{itemize}
\item[\textup{(i)}] The operator $b$ defined in \eqref{LB} is odd, compact and uniformly continuous on bounded sets.
\item[\textup{(ii)}] The funtional $B$ defined in \eqref{BB} is even and compact.
\end{itemize}
\end{Lem}
\begin{proof}
It is obvious that $B$ is even and $b$ is odd. First we verify the uniformly continuity of $b$ in bounded set.
Let $u_1, u_0$ be in bounded set in $\mathcal{H}_{p,0}^{1,(\frac{N-1}{p},\frac{N}{p})}(\mathbb{M})$, and set $\delta:=u_1-u_0 \in \mathcal{H}_{p,0}^{1,(\frac{N-1}{p},\frac{N}{p})}(\mathbb{M})$, then
for any $\varphi\in \mathcal{H}_{p,0}^{1,(\frac{N-1}{p},\frac{N}{p})}(\mathbb{M})$ we have that
\[
  |<b(u_1)-b(u_0),\varphi>|=|\int_{\M} x_1^{p+1}x_2^p(|u_0+\delta|^{p-2}(u_0+\delta)-|u_0|^{p-2}u_0)\varphi d\sigma|
\]
where the binomial theorem implies that
\begin{align}\label{binom}
&|u_0+\delta|^{p-2}(u_0+\delta)-|u_0|^{p-2}u_0
=|\sum_{l=1}^{p-2}C_{p-2}^l u_0^{p-2-l}\delta^l+u_0^{p-2}|(u_0+\delta)-|u_0|^{p-2}u_0 \notag \\
\leq&|\sum_{l=1}^{p-2}C_{p-2}^l u_0^{p-1-l}\delta^l|+|\sum_{l=1}^{p-2}C_{p-2}^l u_0^{p-2-l}\delta^{l+1}|+|u_0^{p-2}\delta|
\leq C \sum_{l=1}^{p-1}| u_0^{p-1-l}\delta^l|.
\end{align}
Then applying H\"older inequality and Lemma \ref{compact}, it implies that
\begin{align}\label{Holdercomp}
 &|<b(u_1)-b(u_0),\varphi>|\leq C\sum_{l=1}^{p-1}\int_{\M}|x_1^{p+1}x_2^p u_0^{p-1-l} \delta^l \varphi|d\sigma \notag \\
 &\leq C \sum_{l=1}^{p-1} (\int_{\M}|x_1^{\frac{N-p-1}{p}}x_2^{\frac{N}{p}-1}u_0|^p d\sigma)^{\frac{p-l-1}{p}}
(\int_\M|x_1^{\frac{N-p-1}{p}}x_2^{\frac{N}{p}-1}\delta|^{p})^{\frac{l}{p}}
(\int_\M|x_1^{\frac{N-p-1}{p}}x_2^{\frac{N}{p}-1}\varphi|^{p})^{\frac{1}{p}}\notag \\
&\leq C \big( \sum_{l=1}^{p-1}\|u_0\|_{ \mathcal{H}_{p,0}^{1,(\frac{N-1}{p},\frac{N}{p})}}^{p-1-l}\|\delta\|^l_{ \mathcal{H}_{p,0}^{1,(\frac{N-1}{p},\frac{N}{p})}}\big)\|\varphi\|_{ \mathcal{H}_{p,0}^{1,(\frac{N-1}{p},\frac{N}{p})}}
\end{align}
Due to the assumption that $u_1, u_2$ are in bounded set and $\delta=u_1-u_0$, we have
\begin{equation}\label{buconti}
 \|b(u_1)-b(u_2)\|_{\mathcal{H}_{p}^{-1,(-\frac{N-1}{p},-\frac{N}{p})}}:=\sup_{\varphi}
  \frac{|<b(u_1)-b(u_0),\varphi>|}{\|\varphi\|_{ \mathcal{H}_{p,0}^{1,(\frac{N-1}{p},\frac{N}{p})}}}\leq C \sum_{l=1}^{p-1}\|u_1-u_0\|^l_{\mathcal{H}_{p,0}^{1,(\frac{N-1}{p},\frac{N}{p})}}
\end{equation}
which verifies the uniformly continuity of $b$ in bounded set.

Now we show that $b$ is a compact operator. For $\{u_k\}$ is bounded in $\mathcal{H}_{p,0}^{1,(\frac{N-1}{p},\frac{N}{p})}(\mathbb{M})$,
then there exists a subsequence of $\{u_k\}$ such that $u_k \rightharpoonup u$ in $\mathcal{H}_{p,0}^{1,(\frac{N-1}{p},\frac{N}{p})}(\mathbb{M})$,
as $k\to \infty$. By choosing proper $\gamma_1$ and $\gamma_2$, Lemma \ref{compact} implies that
$u_k \to u$ in $L^{\gamma_1,\gamma_2}_{p}(\M)$, as $k\to \infty$.

Then we claim that there exists a subsequence holding that
\begin{equation}\label{ae}
  x_1^{\frac{N}{p}-\gamma_1}x_2^{\frac{N}{p}-\gamma_2}u_k \to x_1^{\frac{N}{p}-\gamma_1}x_2^{\frac{N}{p}-\gamma_2}u \quad \textup{a.e}\,\,\, \textup{in}\,\,\, \textup{int}\M.
\end{equation}
In fact, there is a subsequence $\{u_{k_j}\}$ such that
$\| u_{k_{j+1}}-u_{k_j}\|_{L^{\gamma_1,\gamma_2}_{p}} \leq \frac{1}{2^j}$, for $j=1,2,...$.
Let
\[
  x_1^{\frac{N}{p}-\gamma_1}x_2^{\frac{N}{p}-\gamma_2}v_k=
\sum_{j=1}^k|x_1^{\frac{N}{p}-\gamma_1}x_2^{\frac{N}{p}-\gamma_2}u_{k_{j+1}}-x_1^{\frac{N}{p}-\gamma_1}x_2^{\frac{N}{p}-\gamma_2}u_{k_j}|,
\] then  Minkowski inequality gives that
\[
  \|v_k\|_{L^{\gamma_1,\gamma_2}_{p}}\leq \sum_{j=1}^k\|u_{k_{j+1}}-u_{k_j}\|_{L^{\gamma_1,\gamma_2}_{p}}\leq 1.
\]
We set $x_1^{\frac{N}{p}-\gamma_1}x_2^{\frac{N}{p}-\gamma_2}v(x)=\lim_{k\to \infty} x_1^{\frac{N}{p}-\gamma_1}x_2^{\frac{N}{p}-\gamma_2}v_k(x)$.
By Fatou Lemma, it follows that
\[
  \int_\M |x_1^{\frac{N}{p}-\gamma_1}x_2^{\frac{N}{p}-\gamma_2} v(x)|^p d\sigma \leq \liminf_{k\to \infty}\int_\M |x_1^{\frac{N}{p}-\gamma_1}x_2^{\frac{N}{p}-\gamma_2} v_k(x)|^p d\sigma \leq 1
\]
The absolutely convergence implies that
\begin{align*}
x_1^{\frac{N}{p}-\gamma_1}x_2^{\frac{N}{p}-\gamma_2}u_{k_1}+
\sum_{j=1}^k(x_1^{\frac{N}{p}-\gamma_1}x_2^{\frac{N}{p}-\gamma_2}u_{k_{j+1}}-x_1^{\frac{N}{p}-\gamma_1}x_2^{\frac{N}{p}-\gamma_2}u_{k_j}) \to x_1^{\frac{N}{p}-\gamma_1}x_2^{\frac{N}{p}-\gamma_2}u(x)
\end{align*}
a.e in $\M$, which verifies the claim \eqref{ae}.

Then for any $v\in \mathcal{H}_{p,0}^{1,(\frac{N-1}{p},\frac{N}{p})}(\mathbb{M})$, we choose proper $\gamma_1^\prime$, $\gamma_2^\prime$, such that
Lemma \ref{compact} can be applied, it follows that
\begin{align}\label{compbuk}
|&<b(u_k)-b(u),v>|\notag \\
\leq& (\int_\M |x_1^{p+1-(\frac{N}{p}-\gamma^\prime_1)}x_2^{p-(\frac{N}{p}-\gamma^\prime_2)}(|u_k|^{p-2}u_k-|u|^{p-2}u)|^{\frac{p}{p-1}} d\sigma)^{\frac{p-1}{p}}
(\int_\M |x_1^{\frac{N}{p}-\gamma^\prime_1}x_2^{\frac{N}{p}-\gamma_2^\prime} v|^p d\sigma)^{\frac{1}{p}}\notag \\
\leq& C (\int_\M |x_1^{p+1-(\frac{N}{p}-\gamma^\prime_1)}x_2^{p-(\frac{N}{p}-\gamma^\prime_2)}(|u_k|^{p-2}u_k-|u|^{p-2}u)|^{\frac{p}{p-1}} d\sigma)^{\frac{p-1}{p}}
 \|v\|_{\mathcal{H}_{p,0}^{1,(\frac{N-1}{p},\frac{N}{p})}}.
\end{align}
Due to $x_1^{\frac{N}{p}-\gamma_1}x_2^{\frac{N}{p}-\gamma_2}u_k
\to x_1^{\frac{N}{p}-\gamma_1}x_2^{\frac{N}{p}-\gamma_2}u ~ \textup{a.e}~ \textup{in}~\textup{int}\M$, as $k\to \infty$, then we apply Lebesgue dominate
convergence theory to \eqref{compbuk}, and then get the compactness of the operator $b$.

For the compactness of the operator $B$, we take a bounded sequence $\{u_k\}$ in $\mathcal{H}_{p,0}^{1,(\frac{N-1}{p},\frac{N}{p})}(\mathbb{M})$,
then, as before, up to subsequence we have
$ u_k \to u$ in $ L_p^{\frac{N-1}{p}-1,\frac{N}{p}-1}(\M)$.
Then
\begin{align}
B(u_k)=\frac{1}{p}\|u_k\|^p_{L_p^{(\frac{N-1}{p}-1,\frac{N}{p}-1)}}
\to \frac{1}{p}\|u\|^p_{L_p^{(\frac{N-1}{p}-1,\frac{N}{p}-1)}}=B(u)
\end{align}
\end{proof}
The main idea of the proof is to obtain the critical points of $B(u)$ on the manifold
\begin{equation}\label{M}
 M=\{u\in \mathcal{H}_{p,0}^{1,(\frac{N-1}{p},\frac{N}{p})}(\mathbb{M})| \frac{1}{p}\int_\M x_1|\nabla_\M u|^p d\sigma = \alpha \}.
\end{equation}
here $\alpha>0$ is fixed. For each $u\in \mathcal{H}_{p,0}^{1,(\frac{N-1}{p},\frac{N}{p})}(\mathbb{M})\setminus \{0\}$,
we can find $\lambda(u)>0$ such that $\lambda(u) u\in M$, in
the following way
\begin{equation}\label{deflam}
 \lambda(u)=\bigg(\frac{p\alpha}{\int_\M x_1 |\nabla_\M u|^p d\sigma}\bigg)^{\frac{1}{p}}.
\end{equation}
Hence $\lambda: \mathcal{H}_{p,0}^{1,(\frac{N-1}{p},\frac{N}{p})}(\mathbb{M})\setminus \{0\}\to (0,+\infty)$.
It is obvious that $\lambda(u)$ is uniformly continuous on manifold $M$.
By direct computation, the derivative of $\lambda$ is as follows
\begin{equation}\label{deflamp}
 <\lambda^\prime(u), \varphi>=-(p\alpha)^{\frac{1}{p}}\bigg(\int_\M x_1|\nabla_\M u|^p d\sigma\bigg)^{-\frac{p+1}{p}}\int_\M x_1|\nabla_\M u|^{p-2}
 \nabla_\M u\cdot\nabla_\M \varphi d\sigma
\end{equation}
for any $\varphi\in \mathcal{H}_{p,0}^{1,(\frac{N-1}{p},\frac{N}{p})}(\mathbb{M})$.
Therefore, $\int_\M x_1|\nabla_\M u|^{p-2}\nabla_\M u\cdot\nabla_\M \varphi d\sigma=0$ implies $<\lambda^\prime(u), \varphi>=0$.
\begin{Lem}\label{lampunicon}
The functional $\lambda^\prime(\cdot)$ is uniformly continuous on  $M$.
\end{Lem}
\begin{proof}
Let $u_1, u_0$ be in $M$ defined in \eqref{M}, and set $u_1-u_0=:\delta \in \mathcal{H}_{p,0}^{1,(\frac{N-1}{p},\frac{N}{p})}(\mathbb{M})$. For
any $\varphi\in \mathcal{H}_{p,0}^{1,(\frac{N-1}{p},\frac{N}{p})}(\mathbb{M})$, we apply binomial theorem and Lemma \ref{compact} as we did
in \eqref{binom} and \eqref{Holdercomp}, it follows that
\begin{align*}
|<\lambda^\prime(u_1)-\lambda^\prime(u_0),\varphi>|
\leq C\sum_{l=1}^{p-1}\|u_0\|^{p-1-l}_{\mathcal{H}_{p,0}^{1,(\frac{N-1}{p},\frac{N}{p})}}\|\delta\|^{l}_{\mathcal{H}_{p,0}^{1,(\frac{N-1}{p},\frac{N}{p})}}
\|\varphi\|_{\mathcal{H}_{p,0}^{1,(\frac{N-1}{p},\frac{N}{p})}}
\end{align*}
which, as \eqref{buconti}, leads to the uniformly continuity of $\lambda^\prime(\cdot)$ in $M$.
\end{proof}
The next step is to construct a flow on $M$ (defined in \eqref{M}) related to the functional $B(u)$ and the corresponding deformation result allows us to apply the
min-max theory, see \cite{Rab2}. Let $D(u)$ denote the derivative of $B(\lambda(u)u)$ for $u\in \mathcal{H}_{p,0}^{1,(\frac{N-1}{p},\frac{N}{p})}(\mathbb{M})\setminus \{0\}$,
then we have for any $v\in \mathcal{H}_{p,0}^{1,(\frac{N-1}{p},\frac{N}{p})}(\mathbb{M})$
\[
  <D(u),v>=\frac{p\alpha}{\int_\M x_1|\nabla_\M u|^p d\sigma}\bigg(<b(u),v>-\frac{<b(u),u>}{\int_\M x_1|\nabla_\M u|^p d\sigma}
  \int_\M x_1|\nabla_\M u|^{p-2} \nabla_\M u \cdot \nabla_\M v d\sigma\bigg).
\]
where $D(u)\in \mathcal{H}_p^{-1,(-\frac{N-1}{p},-\frac{N}{p})}(\M)$.
If $u\in M$, then
\[
   <D(u),v>=<b(u),v>-\frac{<b(u),u>}{\int_\M x_1|\nabla_\M u|^p d\sigma}
  \int_\M x_1|\nabla_\M u|^{p-2} \nabla_\M u \cdot \nabla_\M v d\sigma.
\]
We claim that $D(u)$ is uniformly continuous in $M$. Since $b(u)$ and $\int_\M x_1|\nabla_\M u|^{p-2}\nabla_\M u\cdot \nabla_\M(\cdot)d\sigma$ are uniformly continuous on $M$ as proved in
Lemma \ref{Bsprop} and Lemma \ref{lampunicon}, then it is sufficient to verify that $<b(u),u>$ hold this property on $M$. In fact, let $u_1, u_0 \in M$,
and set $\delta:=u_1-u_0\in \mathcal{H}_{p,0}^{1,(\frac{N-1}{p},\frac{N}{p})}(\mathbb{M})$. Applying the binomial theorem, H\"older inequality and Lemma \ref{compact} as in \eqref{binom} and \eqref{Holdercomp}, we obtain that
\begin{align*}
|<b(u_1),u_1>-<b(u_0),u_0>|
\leq C \sum_{l=1}^p \|u_0\|^{p-1}_{\mathcal{H}_{p,0}^{1,(\frac{N-1}{p},\frac{N}{p})}}\|u_1-u_0\|^l_{\mathcal{H}_{p,0}^{1,(\frac{N-1}{p},\frac{N}{p})}}.
\end{align*}
which implies the uniformly continuity of $<b(u), u>$ and $D(u)$ on $M$.
Recall the definition of duality map.
\begin{Def}\label{dualmap}
Let $E$ be normed vector space, $E^*$ be the dual space of $E$. We set for every $x_0\in E$
\[
  \mathcal{J}(x_0)=\{ f_0\in E^*; \|f_0\|_{E^*}=\|x_0\|_E ~\mbox{and} ~<f_0,x_0>=\|x_0\|^2\}.
\] The map $x_0 \mapsto \mathcal{J}(x_0)$ is called the duality map from $E$ into $E^*$.
\end{Def}
According to the information of duality map in Chapter 1, \cite{Brez}, here define the duality map
\begin{equation}\label{bigJ}
 \mathcal{J}: \mathcal{H}^{-1,(-\frac{N-1}{p},-\frac{N}{p})}_p(\M) \to \mathcal{H}_{p,0}^{1,(\frac{N-1}{p},\frac{N}{p})}(\mathbb{M})
\end{equation}
 for all $f\in \mathcal{H}^{-1,(-\frac{N-1}{p},-\frac{N}{p})}_p(\M)$, such that $\mathcal{J}$ verifies
\begin{itemize}
\item[\textup{(i)}] $\|\mathcal{J}(f)\|_{\mathcal{H}_{p,0}^{1,(\frac{N-1}{p},\frac{N}{p})}(\mathbb{M})}=\|f\|_{\mathcal{H}^{-1,(-\frac{N-1}{p},-\frac{N}{p})}_p(\M)},$
\item[\textup{(ii)}]$<f,\mathcal{J}(f)>=\|f\|^2_{\mathcal{H}^{-1,(-\frac{N-1}{p},\frac{N}{p})}_p(\M)}$,
\item[\textup{(iii)}] $\mathcal{J}(\cdot)$ is uniformly continuous on bounded sets.
\end{itemize}
For each $u\in M$, we define the tangent component as follows
\begin{equation}\label{bigT}
 T(u)= \mathcal{J}(D(u))-\frac{\int_\M x_1|\nabla_\M u|^{p-2}\nabla_\M u\cdot \nabla_\M(\mathcal{J}(D(u)))d\sigma}{\int_\M x_1|\nabla_\M u|^pd\sigma} u
\end{equation}
such that
$T: M\to \mathcal{H}^{1,(\frac{N-1}{p},\frac{N}{p})}_{p,0}(\M)$ and
\[
 \int_\M x_1|\nabla_\M u|^{p-2}\nabla_\M u\cdot \nabla_\M(T(u))d\sigma=0
\] which implies that if $u\in M$ then
\begin{equation}\label{lambdapT}
<\lambda^\prime(u), T(u)>=0
\end{equation}
\begin{Lem}\label{Tprop}
The tangent component $T(u)$ processes the following properties
\begin{itemize}
\item[\textup(i)] $T(u)$ is odd,
\item[\textup(ii)] $T(u)$ is uniformly continuous on $M$,
\item[\textup(iii)] $T(u)$ is bounded on $M$.
\end{itemize}
\end{Lem}
\begin{proof}
According to the definition of duality map and the fact that $D(u)$ is odd,
we arrive that $T(u)$ is odd. Since both $D(\cdot)$ and $\mathcal{J}(\cdot)$ are uniformly continuous on bounded set,
one can deduce that $T(u)$ is uniformly continuous on M by applying the very similar procedure as in \eqref{binom} and \eqref{Holdercomp}.

On the manifold $M$, by \eqref{bigT}, we have $\|T(u)\|_{\mathcal{H}^{1,(\frac{N-1}{p},\frac{N}{p})}_{p,0}}\leq  I_1+I_2$, with
\[
  I_1=\|\mathcal{J}(D(u))\|_{\mathcal{H}^{1,(\frac{N-1}{p},\frac{N}{p})}_{p,0}}, \quad I_2=\frac{|\int_\M x_1 |\nabla_\M u|^{p-2}\nabla_\M u\cdot \nabla_\M(\mathcal{J}(D(u)))d\sigma|}{|\int_\M x_1|\nabla_\M u|^pd\sigma|}\|u\|_{\mathcal{H}^{1,(\frac{N-1}{p},\frac{N}{p})}_{p,0}}
\]
By applying H\"older inequality and Lemma \ref{compact}, we obtain that
\[
  I_1=\|\mathcal{J}(D(u))\|_{\mathcal{H}^{1,(\frac{N-1}{p},\frac{N}{p})}_{p,0}}=\|D(u)\|_{\mathcal{H}^{-1,(-\frac{N-1}{p},-\frac{N}{p})}_{p}}\leq C \|u\|^{p-1}_{\mathcal{H}^{1,(\frac{N-1}{p},\frac{N}{p})}_{p,0}}
\]
\begin{align*}
  I_2&\leq \frac{\|u\|_{\mathcal{H}^{1,(\frac{N-1}{p},\frac{N}{p})}_{p,0}}^{p-1}\|\mathcal{J}(Du)\|_{\mathcal{H}^{1,(\frac{N-1}{p},\frac{N}{p})}_{p,0}}}
  {\|u\|_{\mathcal{H}^{1,(\frac{N-1}{p},\frac{N}{p})}_{p,0}}^{p}}\|u\|_{\mathcal{H}^{1,(\frac{N-1}{p},\frac{N}{p})}_{p,0}}
  &=\|\mathcal{J}(Du)\|_{\mathcal{H}^{1,(\frac{N-1}{p},\frac{N}{p})}_{p,0}}\\
  &=\|D(u)\|_{\mathcal{H}^{-1,(-\frac{N-1}{p},-\frac{N}{p})}_{p}}\leq C \|u\|^{p-1}_{\mathcal{H}^{1,(\frac{N-1}{p},\frac{N}{p})}_{p,0}}
\end{align*} Then we have that $T(u)$ is bounded on $M$.
\end{proof}
For all $u\in M$, there exists $\gamma_0>0$ and $t_0>0$ such that for all $(u,t)\in M\times [-t_0,t_0]$ it holds
$  \|u+tT(u)\|_{\mathcal{H}^{1,(\frac{N-1}{p},\frac{N}{p})}_{p,0}}\geq \gamma_0>0$.
As a consequence we define the flow
\begin{equation}\label{sigma}
\sigma(u,t):=\lambda(u+tT(u))\,(u+tT(u)) : M\times [-t_0,t_0] \to M
\end{equation}
Then $\sigma(u,t)$ verifies the following properties,
\begin{itemize}
\item[\textup(i)] $\sigma(u,t)$ is odd w.r.t $u$ for fixed $t$;
\item[\textup(ii)] $\sigma(u,t)$ is uniformly continuous with respect to $u$ on $M$;
\item[\textup(iii)] $\sigma(u,0)=u$ for $u\in M$.
\end{itemize}
Indeed, it is obvious that the properties (i) and (iii) of $\sigma(u,t)$ are hold. The uniformly continuity of $\sigma(u,t)$
can be induced from the uniformly continuity of both $\lambda(\cdot)$ and $T(\cdot)$.

In order to obtain the deformation result, we first discover the relation between the functional $B(u)$
and the flow $\sigma(u,t)$ on $M$.
\begin{Lem}\label{rus}
Let $\sigma(u,t)$ be defined in \eqref{sigma}. Then there exists
\[r: M\times [-t_0, t_0] \to \R\] such that $\lim_{\tau \to 0} r(u,\tau) =0$ uniformly on $M$ and
\begin{equation}\label{Bsigma}
B(\sigma(u,t))-B(u)=\int_0^t (\|D(u)\|_{\mathcal{H}^{-1,(-\frac{N-1}{p},-\frac{N}{p})}_{p}}^2+r(u,s))ds
\end{equation}
 for all $u \in M$, and $t\in [-t_0,t_0]$.
\end{Lem}
\begin{proof}
Since $\sigma(u,0)=u$, then $B(u)=B(\sigma(u,0))$. By the definitions of functional $B$ in \eqref{BB} and
the operator $b$ in \eqref{LB}, we have for any $v\in \mathcal{H}^{1,(\frac{N-1}{p},\frac{N}{p})}_{p,0}(\B)$,
\[
  <B^\prime(u), v>=<b(u),v>.
\]
Hence,
\[
  B(\sigma(u,t))-B(u)=\int_0^t <b(\sigma(u,s)),\partial_s \sigma(u,s)> ds.
\]
Due to the fact that $<\lambda^\prime(u), T(u)>=0$ in \eqref{lambdapT} and $\lambda(u)=1$ on $M$ by \eqref{M}, one can derive
\begin{align*}
&\partial_s \sigma(u,s)
=\partial_s(\lambda(u+sT(u))(u+sT(u)))\\
=&<\lambda^\prime(u+sT(u)),T(u)>(u+sT(u))+\lambda(u+sT(u))T(u)\\
=&<\lambda^\prime(u+sT(u))-\lambda^\prime(u),T(u)>(u+sT(u))+(\lambda(u+sT(u))-\lambda(u))T(u)+T(u)\\
:=&R(u,s)+T(u),
\end{align*}
where \[
 R(u,s)=<\lambda^\prime(u+sT(u))-\lambda^\prime(u),T(u)>(u+sT(u))+(\lambda(u+sT(u))-\lambda(u))T(u).
\]

Because $T$ is bounded on $M$, and both $\lambda(u)$ and $\lambda^\prime(u)$ are uniformly
continuous, we have $\lim_{s\to 0} R(u,s)=0$ uniformly on $M$.
Therefore,
\begin{align*}
B(\sigma(u,t))-B(u)&=\int_0^t <b(\sigma(u,s)),R(u,s)+T(u)> ds\\&:=\int_0^t <b(u),T(u)>+r(u,s) ds
\end{align*}
where $ r(u,s)=<b(\sigma(u,s))-b(u),R(u,s)+T(u)>+<b(u),R(u,s)>$.

Since $b$ is uniformly continuous as proved in Lemma \ref{Bsprop}, and the properties that
$\lim_{s\to 0} \sigma(u,s)=u$ and $\lim_{s\to 0} R(u,s)=0$ leads to that
\[
  \lim_{s\to 0} r(u,s)=0
\] uniformly on $M$.
Moreover, a direct computation implies that
\begin{align*}
&<b(u), T(u)>\\
=&<b(u), \mathcal{J}(D(u))>-\frac{<b(u),u>\int_\M x_1|\nabla_\M u|^{p-2}\nabla_\M u\cdot \nabla_\M(\mathcal{J}(D(u)))d\sigma}{\int_\M x_1|\nabla_\M u|^pd\sigma}\\
=&<D(u),\mathcal{J}(D(u))>
=\|D(u)\|^2_{\mathcal{H}^{-1,(-\frac{N-1}{p},\frac{N}{p})}_{p}}
\end{align*}
which verifies \eqref{Bsigma}.
\end{proof}
Consider the level set, for $\beta>0$
\begin{equation}\label{bigphibeta}
\Phi_\beta=\{u\in M~|~ B(u)\geq \beta\}.
\end{equation}
Then we have the following deformation result
\begin{Lem}\label{deform}
Let $\beta>0$ be fixed. Assume that there exists an open set $U\subset M$ such that for some constants $\delta>0$,
$0<\rho<\beta$, it holds that
\[
  \|D(u)\|_{\mathcal{H}^{-1,(-\frac{N-1}{p},-\frac{N}{p})}_{p}}\geq \delta \quad \textup{if}\quad u\in V_\rho=\{u\in M~|~u\notin U, |B(u)-\beta|\leq \rho\}.
\]
Then there exists $\varepsilon >0$ and an operator $\eta_\varepsilon$ such that
\begin{itemize}
\item[\textup{(i)}] $\eta_\varepsilon$ is odd and continuous
\item[\textup{(ii)}] $\eta_\varepsilon(\Phi_{\beta-\varepsilon}-U)\subset \Phi_{\beta+\varepsilon}$.
\end{itemize}
\end{Lem}
\begin{proof}
Take $t_0$ and $r(u,s)$ as in Lemma \ref{rus}. Consider $t_1\in [0,t_0]$, such that for $s\in [-t_1,t_1]$
$|r(u,s)|\leq \frac{1}{2}\delta^2$,
for all $u\in M$.
Then for $u\in V_\rho$ and $t\in [0,t_1]$, we have
\begin{align}\label{3-5}
B(\sigma(u,t))-B(u)&=\int_0^t (\|D(u)\|_{\mathcal{H}^{-1,(-\frac{N-1}{p},-\frac{N}{p})}_{p}}^2+r(u,s))ds\notag \\ &\geq \int_0^t (\delta^2-\frac{1}{2}\delta^2) ds=\frac{1}{2}\delta^2 t.
\end{align}
Choosing $\varepsilon=\min\{\rho,\frac{1}{4}\delta^2 t_1\}$. If $u\in V_\rho\cap \Phi_{\beta-\varepsilon}$, then
$|B(u)-\beta|\leq \rho$,
and from \eqref{3-5}
we have
\begin{equation}\label{t1}
B(\sigma(u,t_1))\geq B(u)+\frac{1}{2}\delta^2 t_1 \geq \beta+\varepsilon.
\end{equation}
By Lemma \ref{rus}, fixing  $u\in V_\rho$, the functional $B(\sigma(u,\cdot))$ is increasing in some interval $[0,s_0)\subset [0,t_1)$.
Then for
\[
  u\in V_\varepsilon =\{u\in M ~|~ u\notin U, |B(u)-\beta|\leq \varepsilon \}
\]
the functional
\begin{equation}\label{tepsilon}
 t_\varepsilon(u)=\min\{t\geq 0|B(\sigma(u,t))=\beta+\varepsilon\}
\end{equation}
is well defined. The inequality \eqref{t1} implies  $0<t_\varepsilon(u)\leq t_1$. The continuity of $\sigma(\cdot,s)$ and the continuity
of $B(\cdot)$ induce that $t_\varepsilon(u)$ is continuous in $V_\varepsilon$.

Define
\begin{equation}
\eta_\varepsilon(u)=\big\{\begin{array}{ll}
\sigma(u,t_\varepsilon(u)) &\textup{if}\quad  u\in V_\varepsilon\\
u& \textup{if}\quad u\in \Phi_{\beta-\varepsilon}-(U\cup V_\varepsilon)
\end{array}
\end{equation}
such that
\[
  \eta_\varepsilon : \Phi_{\beta-\varepsilon}-U \to \Phi_{\beta+\varepsilon}.
\]
Since $\sigma(u,t)$ is odd and uniformly continuous w.r.t $u$, then we have $\eta_\varepsilon (u)$ is odd and continuous.
\end{proof}
We now prove the existence of a sequence of critical values and critical points by applying a min-max argument.
For each $k\in \N$, consider the class
\begin{equation}\label{Ak}
\mathcal{A}_k=\{A\subset M~|~ A~\textup{closed}, A=-A,\gamma(A)\geq k \}
\end{equation}
where $\gamma$ is the genus as in Definition \ref{genus}.
\begin{Lem}\label{seq}
Let $\mathcal{A}_k$ be defined in \eqref{Ak}, define $\beta_k$ as follows
\begin{equation}\label{betak}
\beta_k=\sup_{A\in \mathcal{A}_k}\min_{u\in A}B(u),
\end{equation}
then for each $k$, $\beta_k>0$, and there exists a sequence $\{u_{k_j}\}\subset M$ such that as $j\to \infty$ it holds that
\begin{equation}\label{ukj}
\bigg\{\begin{array}{ll}
\textup{(i)} ~ B(u_{k_j}) \to \beta_k,\\
\textup{(ii)}~ D(u_{k_j}) \to 0.
\end{array}
\end{equation}
\end{Lem}
\begin{proof}
By Definition \ref{genus}, for the manifold $M$ as in \eqref{M}, $\gamma(M)=+\infty$. Hence it holds that $\mathcal{A}_k\neq \emptyset$ for all $k>0$.
For each $k$, given $A\in \mathcal{A}_k$, we have
$\min_{u\in A} B(u)>0$, which implies that $\beta_k >0$ for all $k$.

Assume there is no sequence in $M$ verifying the conditions \eqref{ukj}, then there must exists constants $\delta>0$, $\rho>0$
such that
\[
  \|D(u)\|_{\mathcal{H}^{-1,(-\frac{N-1}{p},-\frac{N}{p})}_{p}}\geq \delta \quad \textup{if}\quad u\in \{u\in M ~|~ |B(u)-\beta_k|\leq \rho\}.
\]
Without loss of generality, assume $\delta<\beta_k$. Applying Lemma \ref{deform} with $U=\emptyset$,
there exists $\varepsilon >0$ and an odd continuous mapping $\eta_\varepsilon$ such that
\[
  \eta_\varepsilon(\Phi_{\beta_k-\varepsilon})\subset \Phi_{\beta_k+\varepsilon}.
\]
By the definition of $\beta_k$ in \eqref{betak}, there exists a set $A_\varepsilon \in \mathcal{A}_k$ such that
\[
  B(u)\geq \beta_k -\varepsilon \quad \textup{in}\quad A_\varepsilon
\]
namely, $A_\varepsilon \subset \Phi_{\beta_k-\varepsilon}$.
Then $B(u)\geq \beta_{k+\varepsilon}$ in $\eta_\varepsilon(A_\varepsilon)$. Since $A_\varepsilon\in \mathcal{A}_k$, then $\gamma(A_\varepsilon)\geq k$.
By Proposition \ref{genusp}, and the fact that $\eta_\varepsilon$ is odd and continuous, we get
\[
  \gamma(\eta_\varepsilon(A_\varepsilon)) \geq k
\]
which implies that
\[
  \eta_\varepsilon(A_\varepsilon)\in \mathcal{A}_k
\]
This is a contradiction with the definition of $\beta_k$ in \eqref{betak}. In this way, for each $k$, we obtain the sequence $\{u_{k_j}\}\subset M$
verifying the conditions \eqref{ukj}.
\end{proof}
To the end, we need the following local $(PS)$ condition.
\begin{Lem}\label{PS}
Let $\{u_j\}\subset M$, $\beta>0$ such that as $j\to \infty$
\begin{equation}\label{uj}
\bigg\{\begin{array}{ll}
\textup{(i)} ~ B(u_{j}) \to \beta,\\
\textup{(ii)}~ D(u_{j}) \to 0, \quad \textup{in} \quad \mathcal{H}^{-1,(-\frac{N-1}{p},-\frac{N}{p})}_{p}(\B).
\end{array}
\end{equation}
Then there exists a convergent subsequence of $\{u_j\}$ in $M$.
\end{Lem}
\begin{proof}
Apply the similar process in the proof of Lemma \ref{PScond}.
\end{proof}
Combining Lemma \ref{seq} and Lemma \ref{PS}, then for each $k$, we have a sequence $u_{k_j}\subset M$ such that
$u_{k_j} \to u_k$ in  $M$ which gives that $u_k \in M$ with $B(u_k)=\beta_k$ and $D(u_k)=0$.

 This induces that
for any $\varphi\in \mathcal{H}^{1,(\frac{N-1}{p},\frac{N}{p})}_{p,0}(\M)$, and for each $k\in \N$
\[
  \int_\M x_1|\nabla_\M u_k|^{p-2}\nabla_\M u_k\cdot\nabla_\M \varphi~ d\sigma=\lambda_k \int_\M x_1^{p+1}x_2^p |u_k|^{p-2}u_k \varphi~ d\sigma
\]
by setting $\lambda_k=\frac{\alpha}{\beta_k}$. This completes the proof of Theorem \ref{T1}.


\subsection{The proof of Theorem \ref{C1}}

Consider $\{E_k\}$ be a sequence of linear subspaces of $\mathcal{H}^{1,(\frac{N-1}{p},\frac{N}{p})}_{p,0}(\M)$, such that $E_k \subset E_{k+1}$; $\overline{\mathcal{L}(\cup_k E_k)}=\mathcal{H}^{1,(\frac{N-1}{p},\frac{N}{p})}_{p,0}(\M)$ and $\textup{dim}E_k =k$.
Define \begin{equation}\label{betatil}
  \tilde{\beta_k}=\sup_{A\in \mathcal{A}_k}\inf_{u\in A\cap E_{k-1}^c} B(u)
\end{equation}
where $E_k^c$ is the linear and topological complementary of $E_k$. It is obvious that
$\tilde{\beta}_k\geq \beta_k >0$.

Hence it is sufficient to show that $\lim_{k\to \infty} \tilde{\beta_k}=0$, which will be verified by
contradiction as follows. Assume for some positive constant $\gamma>0$, we have $\tilde{\beta_k}>\gamma>0$
for all $k\in \N$. Then for each $k\in \N$, there exists $A_k\in \mathcal{A}_k$ such that
\[
  \tilde{\beta}_k \geq \inf_{u\in A_k\cap E_{k-1}^c} B(u)>\gamma.
\]
Then there exists $u\in A_k\cap E_{k-1}^c$ such that
$\tilde{\beta}_k \geq B(u_k)>\gamma$.

In this way, we have formed a sequence $\{u_k\}\subset M$, such that $B(u_k)>\gamma$ for all $k\in \N$.
Since $\{u_k\}\subset M$, as before we know that $\{u_k\}$ is bounded in $\mathcal{H}^{1,\frac{N}{p}}_{p,0}(\B)$,
which implies that there exist $v\in \mathcal{H}^{1,(\frac{N-1}{p},\frac{N}{p})}_{p,0}(\M)$ such that
$ u_k \rightharpoonup v$ in $\mathcal{H}^{1,(\frac{N-1}{p},\frac{N}{p})}_{p,0}(\M)$ and $u_k \to v$ in $L_p^{(\frac{N-p-1}{p},\frac{N-p}{p})}$

Hence we have
\begin{equation}\label{Bgam}
B(v)=\frac{1}{p}\|v\|^p_{L^{\frac{N}{p}-1}_p} >\gamma.
\end{equation}
But the fact that $u_k\in E_{k-1}^c$ implies $v=0$ which induces the contradiction with \eqref{Bgam},
and then finishes the proof of Theorem \ref{C1}.


\end{document}